\documentclass{article}
\usepackage{url}
\usepackage{amsmath}
\usepackage{amsthm}
\usepackage{color,graphicx,fancybox,texdraw}
\usepackage{amssymb}
\usepackage[all]{xy}
\theoremstyle{plain}
\newtheorem*{theorem*}{Theorem}
\newtheorem{theorem}{Theorem}
\newtheorem{lemma}{Lemma}
\newtheorem{proposition}[lemma]{Proposition}

\newtheorem{corollary*}[lemma]{Corollary}

\theoremstyle{definition}
\newtheorem{definition}[lemma]{Definition}
\newtheorem{example}[lemma]{Example}

\numberwithin{equation}{section}
\numberwithin{lemma}{section}

\newcommand{\id}[1]{\left\langle#1\right\rangle}

\newcommand{\F}{\mathbb{F}}

\newcommand{\ROne}{{\F_4[v;\sigma][u]\over\id{u^2+v^2,uv}}}
\newcommand{\RTwo}{{\F_2\id{u,v,w}\over\id{\id{u}^2,\id{v}^2,\id{w}^2,uv,vw,wu,uwv+vuw,uwv+wvu}}}
\newcommand{\RThree}{{\F_4[u][v;\psi]\over\id{u^2,v^2}}}

\newcommand{\diaFinite}{~\xymatrixrowsep{1pc}\xymatrixcolsep{1pc}\xymatrix{
&&&\textrm{reflexive}\ar@{-}[dr]\\
&\textrm{commutative}\ar@{=>}[r]\ar@{=>}[dr]&\textrm{symmetric}\ar@{=>}[r]&\textrm{reversible}\ar@{=>}[d]\ar@{=>}[u]&+\ar@{=>}[l]\\
&&\textrm{duo}\ar@{=>}[r]&\textrm{semicommutative}\ar@{=>}[r]\ar@{-}[ur]&\textrm{abelian}\ar@{=>}[r]&\textrm{2-primal}
}}

\newcommand{\diaReflexive}{~\xymatrixrowsep{1pc}\xymatrixcolsep{1pc}\xymatrix{
\textrm{commutative}\ar@{=>}[r]\ar@{=>}[dr]&\textrm{symmetric}\ar@{=>}[r]&\textrm{reversible}\\
&\textrm{duo}\ar@{=>}[r]&\textrm{semicommutative}\ar@{<=>}[u]\ar@{=>}[r]&\textrm{abelian}\ar@{=>}[r]&\textrm{2-primal}
}}

\newcommand{\dia}{~\xymatrixrowsep{1pc}\xymatrixcolsep{1pc}\xymatrix{
&&\textrm{refl.}\ar@{-}[r]&+\ar@{=>}[dl]\\
\textrm{red.}\ar@{=>}[r]\ar@{=>}@/_/[dr]_{finite}&\textrm{symm.}\ar@{=>}[r]&\textrm{rev.}\ar@{=>}[u]\ar@{=>}[r]&\textrm{semicomm.}\ar@{-}[u]\ar@{=>}[r]\ar@{=>}[drr]&\textrm{PS~I}\ar@{=>}[r]&\textrm{2-pmal}\ar@{<=}@/^2pc/[r]^{artinian}\ar@{=>}[r]&\textrm{NI}\ar@{=>}[r]&\textrm{dedekind-finite}\\
&\textrm{comm.}\ar@{=>}[u]\ar@{=>}[r]&\textrm{r.duo}\ar@{=>}[ur]&&&\textrm{abelian}\ar@{=>}[rru]
}}

\newcommand{\diaIReflexive}{~\xymatrixrowsep{1pc}\xymatrixcolsep{.75pc}\xymatrix{
&&\textrm{reflexive}\ar@{-}[dr]\\
\textrm{commutative}\ar@{=>}[r]\ar@{=>}[dr]&\textrm{symmetric}\ar@{=>}[r]&\textrm{reversible}\ar@{=>}[d]\ar@{=>}[u]&+\ar@{=>}[l]\\
&\textrm{duo}\ar@{=>}[r]&\textrm{semicommutative}\ar@{=>}[r]\ar@{-}[ur]&\textrm{abelian}\ar@{=>}[r]&\textrm{NI}
}}

\title{Minimal Reflexive Nonsemicommutative Rings}
\author{Henry Chimal-Dzul\\Department of Mathematics\\Ohio University\\Athens, OH 45701\\hc118813@ohio.edu\bigskip\\ Steve Szabo\\Department of Mathematics and Statistic\\Eastern Kentucky University\\Richmond, KY 40475\\steve.szabo@eku.edu}
\date{}

\begin{document}
\maketitle 
\begin{abstract}
It has recently been shown that a minimal reversible nonsymmetric ring has order 256 answering a questioned original posed in a paper on a taxonomy of 2-primal rings. Answers to similar questions on minimal rings relating to this taxonomy were also answered in a related work. One type of minimal ring that was left out of that report, was a minimal abelian reflexive nonsemicommutative ring. In this work it is shown that a minimal abelian reflexive nonsemicommutative ring is  of order 256 an example of which is $\F_2D_8$. This is a consequence of the other primary result which is that a finite abelian reflexive ring of order $p^k$ for some prime $p$ and $k<8$ is reversible.
\end{abstract}

\section{Introduction}
Reflexive rings were introduced by Mason in \cite{mason_1981} where he defined a reflexive ideal $I$ of a ring $R$ to be an ideal such that for $a,b\in R$, if $aRb\subset I$ then $bRa\subset I$. As with other such properties on ideals, he went on to define a reflexive ring to be a ring where the zero ideal is a reflexive ideal. Reflexive rings were studied in some detail in \cite{kwak_2012} where connections to semiprime rings, quasi-Baer rings and what they call idempotent reflexive rings were established. In \cite{gu_2013} extensions of reflexive rings were studied. A minimal noncommutative reflexive ring was found to be of order 16 in \cite{kim_2011} where they also show that minimal noncommutative reversible ring is also of order 16.

Minimal examples of various types of rings have become important in applications as many fields move to the use of finite rings from finite fields. As an example, in coding theory, the most general class of rings that are useful are finite Frobenius rings. This was justified by Wood in \cite{wood_1999}. So, it is important to be aware of small Frobenius rings. It is also instructive to have more tangible examples for understanding. For instance, an NI ring can be characterized as a ring whose set of nilpotent elements form an ideal. Commutative rings are NI. The ring $U_2(\F_2)$ is a ring of order 8 which is a minimal noncommutative ring. This ring is NI but notice that $M_2(\F_2)$ is not. This is actually a minimal non-NI ring (see \cite{szabo_2019_3}). For information on minimal rings of type duo see \cite{xue_1992}, reversible see \cite{kim_2011}, semicommutative see \cite{xu_1998} and Frobenius non-chain see \cite{martinez_2015}. Many of these minimal rings mentioned turned out to be of order 16. In \cite{derr_1994} it was shown that there are only 13 noncommutative rings of order 16. From their list and the fact that there is only one noncommutative ring of order 8, with some effort, the minimal rings mentioned so far could be found by simply sifting through this handful of rings. Such work becomes difficult when rings of larger orders are involve. With the classification of rings of order $p^5$ in \cite{corbas_2000} and \cite{corbas_2000_2}, it can be seen that such work becomes cumbersome as the classification of finite rings is difficult and that even for small orders, there is an abundance of rings.

A taxonomy of rings related to 2-primal rings was given in \cite{marks_2003} where the question was asked whether or not $\F_2Q_8$ was a minimal reversible nonsymmetric ring. This was answered positively in \cite{szabo_2019_4}. Another such example was also given there, namely
\[
\frac{\F_2\id{u,v}}{\id{u^3,v^3,u^2+v^2+vu,vu^2+uvu+vuv,u^2vu}},
\]
which is non-duo. Since $\F_2Q_8$ is a duo ring, this shows minimality is independent of the duo property. The mentioned taxonomy was expanded in \cite{szabo_2019_3} where minimal examples were provided to show the various ring class inclusions are strict. Many of the examples in \cite{szabo_2019_3} are rings that could not be found by simply observing the known classifications of rings of order up to $p^5$ since they are larger than 32. An example of a minimal abelian reflexive nonsemicommutative ring was the only type in the classification which was not given. That is the main subject of this paper. The main result of the paper is that a finite abelian reflexive ring of order $p^k$ for some prime $p$ and $k<8$ is reversible (Theorem \ref{theo_main}). With this and the ring $\F_2D_8$ in hand, which will be shown to be reflexive nonsemicommutative ring of order 256 (Example \ref{ex_f2d8}), it can be seen that a minimal abelian reflexive nonsemicommutative ring is of order 256 (Theorem \ref{theo_main2}). It should be pointed out that a minimal reflexive nonabelian ring was provided in Example 3.14 in \cite{szabo_2019_3} which is of order 64 and is minimal amongst all reflexive nonsemicommutative rings given Proposition 3.3 in \cite{szabo_2019_3}.

It is interesting that both minimal reversible nonsymmetric rings and minimal abelian reflexive nonsemicommutative rings are of order 256. Even more, $\F_2Q_8$ and $\F_2D_8$ are examples of the rings just mentioned which are the only two noncommutative group algebras over $\F_2$. The connection between these two ring classes will be clear when the proofs to the present results are compared to the results in \cite{szabo_2019_4}. The important connection is that a ring is reversible if and only if it is reflexive and semicommutative, a fact which is exploited to obtain the current results.

Recall that a finite ring is uniquely expressible as direct sum of rings of prime power order for distinct primes. In light of \cite[Proposition 2.3]{szabo_2019_4}, a reversible ring (and so reflexive) is local if and only if it is indecomposable. Moreover, since the direct sum of a collection of reflexive nonsemicommutative rings is reflexive nonsemicommutative, only finite rings of order $p^n$ need to be considered.

In Section \ref{sect_prelim}, definitions of the various ring types discussed are given as well as the ring class inclusions. It also covers some basic results on these ring types. Section \ref{sect_main} has the the first of the main results which shows that a finite abelian reflexive ring of order $p^k$ for some prime $p$ and $k<8$ is reversible. Section \ref{sect_main2} shows that a minimal abelian reflexive nonsemicommutative ring has order 256. Examples of such minimal rings are also given in Section \ref{sect_main2}.

\section{Preliminaries}
\label{sect_prelim}
In this paper all rings are assumed to be associative with identity $1\neq 0$. Given a ring  $R$, $J(R)$ is the Jacobson radical of $R$ and $N(R)$ is the set of nilpotent elements of $R$. In the present paper the following families of rings are considered.

\begin{definition}
\label{def}
A ring $R$ is ...
\begin{enumerate}
\item{\it symmetric} if for all $a,b,c\in R$, $abc=0$ implies $bac=0$.
\item{\it reversible} if for all $a,b\in R$, $ab=0$ implies $ba=0$.
\item{\it right (resp. left) duo} if for all $a,b\in R$, $ba\in aR$ ($ba\in Rb$) (equivalently, every right (left) ideal of $R$ is 2-sided). A ring that is both right and left duo is simply a {\it duo} ring.
\item{\it semicommutative} if for all $a,b\in R$, $ab=0$ implies $aRb=0$.
\item{\it reflexive} if for all $a,b\in R$, $aRb=0$ implies $bRa=0$.
\item{\it abelian} if each idempotent of $R$ is central.
\item{\it NI} if $N(R)\lhd R$.
\end{enumerate}
\end{definition}

Clearly, any symmetric ring is reversible. Also, in view of \cite[Lemma 2.7]{szabo_2019_3}, a ring is reversible if and only if it is reflexive and semicommutative. Direct computations show that one-sided duo rings are semicommutative, and semicommutative rings are abelian (\cite[Lemma 2.5]{szabo_2019_3}). Finally, in the context of finite rings, both right and left duo rings are duo and abelian rings are NI neither of which is the case in general. Figure \ref{fig1} is a diagram taken from \cite{szabo_2019_3} of the various finite ring class inclusions. The class inclusions in Figure \ref{fig1} are strict as has been shown in \cite{hwang_2006,marks_2002,marks_2003,szabo_2019_3}. In particular, in the recent works \cite{szabo_2019_4} and \cite{szabo_2019_3} some examples of minimal rings that support the independence between the families of symmetric, reversible, reflexive, duo and semicommutative rings were presented. The present article is a further contribution in this direction. Of course a search of minimal rings need only consider finite rings.

\begin{center}
\begin{figure}[h!]
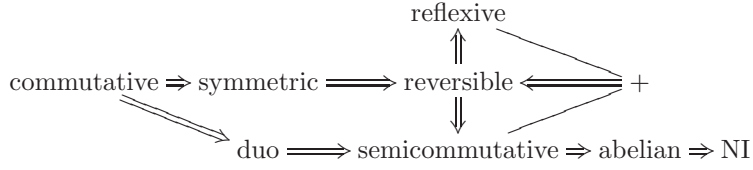

\diaIReflexive
\caption{Finite Rings}
\label{fig1}
\end{figure}
\end{center}

For basic facts about local rings, the reader may consult \cite{lam_2001}. Here some important facts for the purposes of this paper are reviewed. When $R$ is finite, $J=J(R)$ is a nilpotent (two-sided) ideal, so there exists a least integer $K$ such that $J^{K+1}=0$ but $J^{K}\neq 0$. The integer $K+1$ is called the \textit{index of nilpotency} of $J$. The following is a well-known result about finite local rings.

\begin{lemma}
\label{lemma_vspace}
Let $R$ be a finite local ring, $K+1$ the index of nilpotency of $J=J(R)$ and $F=R/J$. Then $F$ is a finite field and for each $1\leq i\leq K$, $J^i/J^{i+1}$ is a finite dimensional $F$-vector space.
\end{lemma}

Using Lemma \ref{lemma_vspace}, it can be shown that in a local ring $R$ with $J=J(R)$, $K+1$ the index of nilpotency of $J$ and $F=R/J$, there exists a minimal set of generators of $J(R)^i$, $\{u_{i1},\dots,u_{id_i}\}\subset J^i\setminus J^{i+1}$ for each $1\leq i\leq K$, and that each element of $R$ can be uniquely expressed in the form
\[
\alpha_0+\sum_{i=1}^K\alpha_{i1}u_{i1}+\dots+\alpha_{id_i}u_{id_i},
\]
where $\alpha_0,\alpha_{ij}\in F$ and $d_i=\dim_{F}J^i/J^{i+1}$. This fact will be used freely throughout the work.

Recall that a ring $R$ is called a \textit{left (right resp.) chain ring} if its lattice of left (right resp.) ideals form a chain under inclusion.  This type of finite ring has been characterized by E. W. Clark and D. A. Drake in \cite[Lemma 1]{clark_1973}. Pertinent facts on finite chain rings are collected in the following lemma.

\begin{lemma}
\label{lemma_chain}
For a finite ring $R$ with $J=J(R)$, the following are equivalent:
\begin{enumerate}
\item $R$ is a left chain ring.
\item $R$ is a local ring and $J$ is principal.
\item there exists $u\in R$ such that every proper left or right ideal has the form $J^i=\langle u^i\rangle$ for some $i\in\mathbb{N}$.
\item $R$ is a right chain ring.
\end{enumerate}
\end{lemma}

\begin{lemma}
\label{lemma_sc1}
Let $R$ be a finite chain ring with $J=J(R)\neq 0$, $\sigma \in S_n$, $x_1,\ldots,x_n\in J\setminus \{0\}$ and $l=\sum_{i=1}^n\alpha_i$, where $x_i\in J^{\alpha_i}\setminus J^{\alpha_i+1}$ for $1\leq i\leq n$. Then $x_1\cdots x_n\in J^{l}$ if and only if $x_{\sigma(1)}\cdots x_{\sigma(n)}\in J^l$.
\end{lemma}
\begin{proof}
By Lemma \ref{lemma_chain}, there exists $u\in J$ such that $J=\id{u}$. Let $K+1$ be the index of nilpotency of $J$ and $x,y\in J\setminus \{0\}$. Then there exist positive integers $1\leq \alpha,\beta\leq K$ such that $x\in J^{\alpha}\setminus J^{\alpha+1}$ and  $y\in J^{\beta}\setminus J^{\beta +1}$. Since $R$ is local, elements not in $J$ are units. Thus, Lemma \ref{lemma_chain} establishes that $x=s_1u^\alpha=u^\alpha s_2$  and  $y=t_1u^{\beta}=u^{\beta}t_2$ for some units $s_i,t_i$, $i=1,2$. Consequently,
\[
xy=(s_1u^\alpha)(u^\beta t_2)=s_1u^{\alpha+\beta}t_2,\textrm{ and }yx=(t_1u^\beta)(u^\alpha s_2)=t_1u^{\alpha+\beta}s_2.
\]
Then $xy$ and $yx$ are in $J^{\alpha+\beta}$. An inductive argument shows that $x_1\cdots x_n\in J^{l}$ if and only if $x_{\sigma(1)}\cdots x_{\sigma(n)}\in J^l$.
\end{proof}

\begin{proposition}
\label{lemma_chainsym}
 A finite chain ring is duo and symmetric.
\end{proposition}
\begin{proof}
A finite chain ring being duo is an immediate consequence of Lemma \ref{lemma_chain}. With Lemma \ref{lemma_sc1}, the fact that a finite chain ring is symmetric can also be easily established.
\end{proof}

To conclude this section, three technical lemmas are given which are needed in the results of Section \ref{sect_main}. Lemma \ref{lemma_x1} establishes the fact that for a finite local ring $R$ if $J(R)^i$ is principal and $J(R)^{i+1}\neq 0$ for some $i$, then $J(R)^j$ is principal for $j\geq i$ and that there is an element of $J(R)\setminus J(R)^2$ that has the same index of nilpotency as $J(R)$.

\begin{lemma}
\label{lemma_x1}
Let $R$ be a finite local ring, $J=J(R)$ and $F=R/J$. Assume there is $n\geq 1$ such that $J^{n}$ is principal and $J^{n+1}\neq 0$. Then there is $u\in J\setminus J^{2}$ such that $J^{l}=\langle u^l\rangle$ for all $l\geq n$.
\end{lemma}
\begin{proof}
Let $u_1,\dots,u_s\in J\setminus J^2$ such that $J=\langle u_1,\dots,u_s\rangle$. Since $J^n$ is principal, there is $x_1,\dots,x_n\in J\setminus J^{2}$ such that $J^n=\langle x_1\cdots x_n\rangle$. Hence $J^{n+1}=\langle u_1x_1\cdots x_n,\dots, u_s x_1\cdots x_n\rangle$. Also, by the fact that $J^{n+1}\neq 0$, then $J^{n+1}\neq J^{n+2}$, so that there is $j$ such that $u_j x_1\cdots x_n\in J^{n+1}\setminus J^{n+2}$. Let $u=u_j$. Using that $J^n/J^{n+1}$ is a 1-dimensional $F$-vector space, there is $\alpha_1\in F\setminus\{0\}$ such that
\[
ux_1\cdots x_{n-1}+J^{n+1}=\alpha_1x_1\cdots x_n+J^{n+1}\neq J^{n+1}.
\]
Then
\[
u^2x_1\cdots x_{n-1}+J^{n+2}=u\alpha_1 x_1\cdots x_n+J^{n+2}\neq J^{n+2}.
 \]
Since $\alpha_1$ is a unit, $u\alpha_1 x_1\cdots x_n\in J^{n+1}\setminus J^{n+2}$ by the choice of $u$. Thus, $u^2x_1\cdots x_{n-1}\in J^{n+1}\setminus J^{n+2}$, where we conclude that $u^2x_1\cdots x_{n-2}\in J^{n}\setminus J^{n+1}$. Again, there is $\alpha_2\in F\setminus\{0\}$ where
\[
u^2x_1\cdots x_{n-2}+J^{n+1}=\alpha_2 x_1\cdots x_n+J^{n+1},
\]
which implies
\[
u^3x_1\cdots x_{n-2}+J^{n+2}=u\alpha_2 x_1\cdots x_n+J^{n+2}.
\]
This leads to the fact $u^3x_1\cdots x_{n-3}\in J^{n}\setminus J^{n+1}$. Continuing in the same manner, it can be shown that $u^n+J^{n+1}=\alpha_nx_1\cdots x_n+J^{n+1}$ for some $\alpha_n\in F\setminus\{0\}$. That is,  $J^n=\langle a^{n}\rangle$.

Now, assume $J^{l}=\langle u^l\rangle$ for some $l\geq n$. Then $J^{l+1}=\langle u_1 u^l,\dots, u_s u^l\rangle$. For each $j$, $u_j u^{l-1}\in J^{l}$, so there is $r_j\in R$ such that $u_ju^{l-1}=r_ju^{l}$. Thus,
\[
u_ju^l=(u_ju^{l-1})u=(r_ju^l)u=r_ju^{l+1}.
\]
It follows that $J^{l+1}=\langle u^{l+1}\rangle$.
\end{proof}

\begin{lemma}
\label{lemma_lrs}
Let $R$ be a local ring with $J=J(R)$ and $F=R/J$. Assume $F$ is a prime field. Then $R$ is semicommutative if and only if for any $a,b\in J$ where $ab=0$, $aJb=0$.
\end{lemma}
\begin{proof}
Clearly, if $R$ is semicommutative, the condition holds. Now, assume the condition holds. Let $a,b\in R$ where $ab=0$. If $a$ or $b$ is a unit then the other is 0 and $aRb=0$. So, assume $a,b\in J$. Let $r\in R$. Then $r=r_0+r_1$ for some $r_0\in F$ and $r_1\in J$. Since $F$ is prime,
\[
arb=a(r_0+r_1)b=ar_0b+ar_1b=ar_0b=a(\underbrace{1+\cdots+1}_{r_0-times})b=r_0ab=0.
\]
So, $aRb=0$ and $R$ is semicommutative.
\end{proof}

\begin{lemma}
\label{lemma_lr}
Let $R$ be a local reflexive ring with $J=J(R)$ and $F=R/J$. Assume that $J^4=0$ and $F$ is a prime field. Let $a,b,c\in J$. Then $abc=0$ if and only if $bca=0$ if and only if $cab=0$.
\end{lemma}
\begin{proof}
Let $r\in R$. Then $r=r_0+r_1$ for some $r_0\in F$ and $r_1\in J$. Since $J^4=0$ and $F$ is prime, $ar(bc)=ar_0(bc)=r_0(abc)$. Therefore, $abc=0$ if and only if $aR(bc)=0$. Since the ring is reflexive, then $abc=0$ if and only if $(bc)Ra=0$, if and only if $bca=0$, if and only if $cab=0$.
\end{proof}

In Lemma \ref{lemma_lr}, the idea is that $abc\in J(R)^3$ and the index of nilpotency of $J(R)$ is at most 4. The result can be generalized to say that if $a_1\cdots a_n\in J^K$ where $a_1,\dots,a_n\in R$ and the index of nilpotency of $J(R)$ is at most $K+1$, then $a_1\cdots a_n=0$ if and only if $a_na_1\cdots a_{n-1}=0$.

\section{Small Abelian Reflexive Rings}
\label{sect_main}
The main result of the paper is that a minimal abelian reflexive nonsemicommutative ring is of order 256. In this section it is shown that a local reflexive ring of order $p^k$ with $k<8$ is semicommutative and therefore reversible. This then implies that a local reflexive ring of an order less than $256$ is semicommutative. To that end, the following five results are given which will ultimately be used to show the main theorem of this section.

\begin{proposition}
\label{prop_j3}
Let $R$ be a finite local ring, $J=J(R)$ and $F=R/J$ where $F$ is prime. Assume $J^3=0$. Then $R$ is semicommutative.
\end{proposition}
\begin{proof}
Let $a,b\in J$. Then $aJb=0$. The result follows from Lemma \ref{lemma_lrs}.
\end{proof}

In Theorem 3.2 of \cite{szabo_2019_3} it was shown that
\[
R={\F_2\id{u,v}\over\id{u^3,v^2,vu,u^2-uv}}
\]
is a minimal nonreflexive semicommutative nonduo ring. Note that $R$ is a finite local ring such that $R/J(R)$ is prime and $J(R)^3=0$. This shows that Proposition \ref{prop_j3} can not be strengthened to conclude the ring is reflexive and therefore reversible.

\begin{proposition}
\label{prop_31s}
Let $R$ be a finite local ring, $J=J(R)$ and $F=R/J$ where $F$ is prime. Assume $\dim_F(J/J^2)\leq 3$ and $\dim_F(J^2/J^3)=1$. Then $R$ is semicommutative.
\end{proposition}
\begin{proof}
Let $K+1$ be the index of nilpotency of $J$. If $J^3=0$, by Proposition \ref{prop_j3} $R$ is semicommutative. Assume $J^3\neq0$. Then by Lemma \ref{lemma_x1}, since $\dim_F(J^2/J^3)=1$ there exists $u_1\in J\setminus J^2$ such that for $k>1$, $J^k=Fu_1^k+J^{k+1}$. Furthermore, $J=Fu_1+Fu_2+Fu_3+J^2$ for some $u_2,u_3\in J\setminus J^2$.

Now, $u_iu_1=\sum_{j=2}^K\alpha_j u_1^j$ and $u_1u_i=\sum_{j=2}^K\beta_j u_1^j$ for some $\alpha_j,\beta_j\in F$. Since $F$ is a prime field,
\[
\sum_{j=2}^{K-1}\beta_j u_1^{j+1}=u_1u_iu_1=\sum_{j=2}^{K-1}\alpha_j u_1^{j+1}\\
\]
So, for $2\leq j\leq K-1$, $\alpha_j=\beta_j$ showing $u_iu_1=u_1u_i+t_i$ for some $t_i\in J^K$. Now, $u_2u_3=\sum_{j=2}^K\gamma_j u_1^j$ for some $\gamma_j\in F$. So, $u_iu_2u_3=u_2u_3u_i$ for any $i$. This shows that given any 3-rd degree monomial in $u_1$, $u_2$ and $u_3$, any permutation of the product is the same.

By Lemma \ref{lemma_lrs}, it only needs to be shown that for any $a,b\in J$ such that $ab=0$, $aJb=0$. Let $a,r,b\in J$ such that $ab=0$. By the property on the permutation of triple products of $u_1$, $u_2$ and $u_3$ and the fact that $F$ is prime, $arb=rab=0$. Hence, $R$ is semicommutative.
\end{proof}

The ring
\[
R={\F_2\id{u,v}\over\id{u^4,v^2,vu,u^3-uv}}
\]
is a finite local ring such that $R/J(R)$ is prime, $\dim_F(J/J^2)\leq 3$ and $\dim_F(J^2/J^3)=1$. By Proposition \ref{prop_31s}, $R$ is semicommutative. Since $uv\neq 0$ and $vu=0$, it is nonreversible and therefore nonreflexive. This shows that Proposition \ref{prop_31s} can not be strengthened to include reflexivity in its conclusions.

\begin{proposition}
\label{prop_23}
Let $R$ be a finite local reflexive ring, $J=J(R)$ and $F=R/J$ where $F$ is prime. Assume $\dim_F(J/J^2)=2$ and $J^4=0$. Then $R$ is semicommutative.
\end{proposition}
\begin{proof}
By Lemma \ref{lemma_lrs}, it only needs to be shown that for any $a,b\in J$ such that $ab=0$, $aJb=0$. Let $a,r,b\in J$ such that $ab=0$. If $a$, $r$ or $b$ is in $J^2$ then $arb=0$ since $J^4=0$. So, assume $a,r,b\in J\setminus J^2$. If $a+J^2=\gamma r+J^2$ for $\gamma\in F$ then $arb=\gamma r^2b=rab=0$. So, assume $J=\id{a,r}$. Then $b=\alpha a+\beta r+t$ where $\alpha,\beta\in F$ and $t\in J^2$. If $\alpha=0$ then $arb=\beta ar^2=abr=0$. By Lemma \ref{lemma_lr}, since $rab=0$, if $\beta=0$ then $arb=\alpha ara=bra=rab=0$. Finally, assume $\alpha,\beta\neq 0$. Then, since $ab=0$,
\begin{align*}
0=abr&=\alpha a^2r+\beta ar^2\\
0=a^2b&=\alpha a^3+\beta a^2r\\
0=aba&=\alpha a^3+\beta ara.
\end{align*}
So, $\beta ar^2=-\alpha a^2r=-\alpha ara$ and so $arb=\alpha ara+\beta ar^2=0$. Hence, $R$ is semicommutative.
\end{proof}

In the previous proposition reflexivity was not superfluous as can be seen by noting the ring
\[
{\F_2\id{u,v}\over\id{u^2,v^2,uvu-vuv}}.
\]
This ring is a finite local nonsemicommutative nonreflexive ring (see \cite{szabo_2019_3}) with a Jacobson radical that is 2-generated having index of nilpotency of 4. This example also shows the necessity of reflexivity in the following theorem.

\begin{proposition}
\label{prop_32}
Let $R$ be a finite local reflexive ring, $J=J(R)$ and $F=R/J$ where $F$ is prime. Assume $\dim_F(J^2/J^3)\leq 2$ and $J^4=0$. Then $R$ is semicommutative.
\end{proposition}
\begin{proof}
By Lemma \ref{lemma_lrs}, it only needs to be shown that for any $a,b\in J$ such that $ab=0$, $aJb=0$. Note that Lemma \ref{lemma_lr} will be needed frequently throughout this proof so it will be used freely. Let $a,r,b\in J$ such that $ab=0$. If $a$, $r$ or $b$ is in $J^2$ then $arb=0$ since $J^4=0$. So, assume $a,r,b\in J\setminus J^2$. If $ar\in J^3$, $rb\in J^3$ or $ba\in J^3$ then it can be shown that $arb=0$. Assume, $ar,rb,ba\in J^2\setminus J^3$. If $ar+J^3$ or $rb+J^3$ is a nonzero scalar multiple of $ba+J^3$ then $arb=0$. So, assume neither is the case. Then $\dim_F(J^2/J^3)=2$, there exists $\alpha,\beta\in F\setminus\{0\}$ and $\gamma\in F$ such that $\alpha ar+\beta rb+\gamma ba\in J^3$. If $ra\in J^3$ then $0=ara=a^2r$ showing $\beta arb=\alpha a^2r+\beta arb+\gamma aba=a(\alpha ar+\beta rb+\gamma ba)=0$. Since $\beta\neq 0$, $arb=0$ in this case. So assume $ra\in J^2\setminus J^3$. If $ra+J^2=ar+J^2$ then $arb=rab=0$. Assume $J^2=Fra+Far+J^3$. So, $ba+J^2=\delta(ra+J^2)+\zeta(ar+J^2)$ for $\delta,\zeta\in F$. Since the case when $\delta=0$ has already been considered, assume $\delta\neq 0$. Then $0=bab=\delta rab+\zeta arb=\zeta arb$. If $\zeta\neq 0$ then $arb=0$. Finally, assume $\zeta=0$. Then $ba+J^2=\delta ra+J^2$. Then $0=\delta^{-1}aba=ara$ showing $0=(\alpha ar+\beta rb+\gamma ba)a=\beta rba=\beta arb$. Since $\beta\neq 0$, $arb=0$ in this case. Hence, $R$ is semicommutative.
\end{proof}

\begin{proposition}
\label{prop_2211}
Let $R$ be a finite local reflexive ring, $J=J(R)$ and $F=R/J$ where $F$ is prime. Assume $\dim_F(J/J^2)=2$, $\dim_F(J^2/J^3)=2$, $\dim_F(J^3/J^4)=1$, $J^4\neq 0$ and $J^5=0$. Then $R$ is semicommutative.
\end{proposition}
\begin{proof}
By Lemma \ref{lemma_x1}, $J=Fu+Fv+J^2$, $J^2=Fu^2+Fw+J^3$, $J^3=Fu^3+J^4$ and $J^4=Fu^4$ for $u,v\in J\setminus J^2$ and $w\in\{uv,vu,v^2\}$. For some $\lambda\in F$, $vu^3=\lambda u^4$ and so $(v-\lambda u)u^3=0$. Replace $v$ with $v-\lambda u$. Since $F$ is prime and $J^5=0$, $vRu^3=0$. Since $R$ is reflexive, $u^3Rv=0$ and $u^3v=0$. Using the reflexivity of $R$, it can be shown then that any 4-th degree monomial involving $v$ is 0. Then any 3-rd degree monomial involving $v$ is annihilated by both $u$ and $v$ on both the right and the left showing these monomials are in $J^4$. Given $\alpha u^2+\beta uv+\gamma uv+\delta u^2\in J^3$, $u(\alpha u^2+\beta uv+\gamma uv+\delta u^2)\in J^4$ showing $\alpha=0$ since $u^3\in J^3\setminus J^4$. This shows that $uv+J^3$, $vu+J^3$ and $v^2+J^3$ are each multiples of $w+J^3$.

Now, $uvu=\alpha u^4$ for some $\alpha\in F$. So, $u(vu-\alpha u^3)=0$. By reflexivity, $(vu-\alpha u^3)u=0$ and then $vu^2=\alpha u^4=uvu$. Similarly, $u^2v=uvu$ so, $u^2v=uvu=vu^2$. Also, $v^2u=\beta u^4$  for some $\beta\in F$. So, $(v^2-\beta u^3)u=0$. By reflexivity, $u(v^2-\beta u^3)=0$ and then $uv^2=\beta u^4=v^2u$. This shows that $u$ commutes with any element of $J^2$.

If $vu\in J^3$ or $uv\in J^3$ then $vuv=0$ and by reflexivity, $uv^2=v^2u=vuv=0$. Alternatively, if it were assumed that $v^2\in J^2\setminus J^3$, then $w=v^2$. So, $uv=\delta v^2+t$ for some $\delta\in F$ and $t\in J^3$ and then $vuv=\delta v^3=uv^2$. Then by reflexivity again, $vuv=uv^2=v^2u$. If instead $u^2v\neq 0$ then $vuv=\gamma u^2v$ for some $\gamma\in F$, so $v^2u=\gamma vu^2=\gamma u^2v=vuv$. Hence, $v^2u=vuv=uv^2$ in this case as well. Clearly, if $v^2\in J^4$, then $v^2u=vuv=uv^2$. In any of these cases, $v$ commutes with any element of $J^2$. Next, it is shown that if none of these situations arise, the ring is not reflexive.

Assume $uvu=0$, $vuv\neq 0$, $J^2=\id{u^2,uv}$, $vu=\alpha_1uv+\alpha_2u^3+\alpha_3u^4$ and $v^2=\gamma_1u^3+\gamma_2u^4$ where $\alpha_1\neq 0$ and $\gamma_1\neq 0$. Then $0=uvu=\alpha_1u^2v+\alpha_2u^4+\alpha_3u^5=\alpha_2u^4$, so, $\alpha_2=0$. Then $vu=\alpha_1uv+\alpha_3u^4$ and $vuv=\alpha_1uv^2=\alpha_1^2vuv$. So, $\alpha_1^2=1$. If $\alpha_1=1$, then $vuv=uv^2$ so assume $\alpha_1\neq 1$ as well. It will be shown in this case that $R$ is not reflexive. Let $a=uv+u^3$ and $b=u-\gamma_1^{-1}v$. Let $r\in R$ where $r=r_0+r_1$ with $r_0\in F$ and $r_1\in J$. Since $aJ\subset J^4$,
\begin{eqnarray*}
arb&=&ar_0b+ar_1b=ar_0b=(uv+u^3)r_0(u-\gamma_1^{-1}v)\\
&=&r_0(u^4-\gamma_1^{-1}uv^2)=r_0(u^4-\gamma_1^{-1}\gamma_1u^4)=r_0(u^4-u^4)\\
&=&0
\end{eqnarray*}
but
\begin{eqnarray*}
ba&=&(u-\gamma_1^{-1}v)(uv+u^3)=(u^4-\gamma_1^{-1}vuv)(u^4-\gamma_1^{-1}\alpha_1\gamma_1u^4)\\
&=&(u^4-\alpha_1u^4)=(1-\alpha_1)u^4\\
&\neq&0.
\end{eqnarray*}
Since $R$ is reflexive, this is not true. Therefore, in any case, $u$ and $v$ commute with the elements of $J^2$. So, any product of three elements from $J$ is equal to any permutation of the product. Thus, by Lemma \ref{lemma_lr}, $R$ is semicommutative.
\end{proof}

The five results in the section can be collected in the following way. Let $R$ be a finite local ring, $J=J(R)$ and $F=R/J$ where $F$ is prime. If any of the following are true then $R$ is semicommutative.
\begin{itemize}
\item $J^3=0$.
\item $\dim_F(J/J^2)\leq 3$ and $\dim_F(J^2/J^3)=1$.
\item $R$ is reflexive, $\dim_F(J/J^2)=2$ and $J^4=0$.
\item $R$ is reflexive, $\dim_F(J^2/J^3)\leq 2$ and $J^4=0$.
\item $R$ is reflexive, $\dim_F(J/J^2)=2$, $\dim_F(J^2/J^3)=2$, $\dim_F(J^3/J^4)=1$, $J^4\neq 0$ and $J^5=0$.
\end{itemize}

\begin{theorem}
\label{theo_main}
A finite abelian reflexive ring of order $p^k$ for some prime $p$ and $k<8$ is reversible.
\end{theorem}
\begin{proof}
Let $R$ be a finite local reflexive nonsemicommutative ring of order $p^l$ for some prime $p$, $J=J(R)$, $K+1$ be the index of nilpotency of $J$ and $F=R/J$. Since $R$ is finite and local, by Lemma \ref{lemma_vspace},  $F$ is a field and $|R|=\left|F\right|^r$ for some $r$. Since $R$ is nonsemicommutative it is nonsymmetric. Then by Proposition \ref{lemma_chainsym}, $R$ is not a chain ring. Then $r>2$. Now, $r\neq 3$, because otherwise since $R$ is not a chain ring, $J^2=0$ meaning $R$ is semicommutative which it is not. So, $r\geq 4$. Then $|R|\geq|F|^4$. If $|F|$ is not prime, $l\geq 8$. So, assume $|F|$ is prime. By Proposition \ref{prop_j3}, $J^3\neq 0$. By Proposition \ref{prop_31s}, if $\dim_F(J/J^2)\leq 3$ then $\dim_F(J^2/J^3)>1$. Let $d_i=\dim_F(J^i/J^{i+1})$ and $D=(d_1,d_2,\dots,d_K)$. If $l<8$, then the only possibilities for $D$ are $(2,2,1)$, $(2,2,2)$, $(2,3,1)$, $(3,2,1)$, $(4,1,1)$ and $(2,2,1,1)$. By Proposition \ref{prop_23}, $D\notin\{(2,2,1),(2,2,2),(2,3,1)\}$. By Proposition \ref{prop_32}, $D\notin\{(3,2,1),(4,1,1)\}$. Finally, by Proposition \ref{prop_2211}, $D\neq(2,2,1,1)$. Hence, $l\geq 8$. Since an abelian reflexive semicommutative ring is a direct sum of local reflexive semicommutative rings, any finite abelian reflexive ring of order $p^k$ for some prime $p$ and $k<8$ is semicommutative. Since reflexive semicommutative rings are reversible, the result follows.
\end{proof}

\section{Minimal Abelian Reflexive Nonsemicommutative Rings}
\label{sect_main2}
\begin{example}
\label{ex_f2d8}
Let $R=\F_2D_8$ where $D_8=\id{r,s|r^4=s^2=(sr)^2=1}$. Let $B=\{1+a|a\in D_8\setminus\{1\}\}$. Then $\{1\}\cup B$ is an $F$-basis for $R$. It can be shown that $\id{B}$ is an ideal of $R$ and that $\id{B}$ is nil. Then any element outside $\id{B}$ is a unit showing $R$ is local. Next, notice $(1+rs)(r+s)=0$ but $(1+rs)s(r+s)\neq0$ so $R$ is not semicommutative. It is now shown that $R$ is reflexive.

Let $S$ be the subspace of $R$ with basis $\{1,r,r^2,r^3\}$ and define $\bar{~}:S\to S$ via $\bar{~}:\alpha_0+\alpha_1r+\alpha_2r^2+\alpha_3r^3\mapsto\alpha_0+\alpha_3r+\alpha_2r^2+\alpha_1r^3$ which is clearly an automorphism on $S$. Let $a=a_1+a_2s,b=b_1+b_2s\in S+Ss=R$. Assume $aRb=0$. So,
\[
0=ab=(a_1+a_2s)(b_1+b_2s)=(a_1b_1+a_2\bar{b_2})+(a_1b_2+a_2\bar{b_1})s
\]
and
\[
0=arb=(a_1+a_2s)r(b_1+b_2s)=(a_1rb_1+a_2r^3\bar{b_2})+(a_1rb_2+a_2r^3\bar{b_1})s.
\]
Then $a_1b_1+a_2\bar{b_2}=0$ and $a_1b_1+a_2r^2\bar{b_2}=0$. So, $a_1b_1=a_2\bar{b_2}$ and $a_1b_1=r^2a_1b_1$. From the last result we can deduce that $a_1b_1=\overline{a_1b_1}$. So, $\bar{a_2}b_2=a_1b_1$ and $a_1b_1$ annihilates $1+r^2$. Similarly, it can be shown that $\bar{a_2}b_1=a_1b_2$ and $a_1b_2$ annihilates $1+r^2$. Using that $asb=0$ and $arsb=0$ it can be deduced that $\bar{a_1}b_2=a_2b_1$, $\bar{a_1}b_1=a_2b_2$, and both $a_2b_1$ and $a_2b_2$ annihilate $1+r^2$. Before proceeding note that for $y\in S$, $y+\bar{y}\in\id{1+r^2}$. To finish, let $x=x_1+x_2s\in S+Ss=R$. Then
\begin{eqnarray*}
bxa&=&(b_1+b_2s)(x_1+x_2s)(a_1+a_2s)\\
&=&b_1x_1a_1+b_1x_2\bar{a_2}+b_2\bar{x_1}\bar{a_2}+b_2\bar{x_2}a_1\\
&&+(b_1x_1a_2+b_1x_2\bar{a_1}+b_2\bar{x_1}\bar{a_1}+b_2\bar{x_2}a_2)s\\
&=&b_1a_1(x_1+\bar{x_1})+b_2a_1(x_2+\bar{x_2})\\
&&+(b_1a_2(x_1+\bar{x_1})+b_2a_2(x_2+\bar{x_2}))s\\
&=&0.
\end{eqnarray*}
Hence, $R$ is a local reflexive nonsemicommutative ring of order 256.
\end{example}

By Theorem \ref{theo_main}, any abelian reflexive ring of order less than 256 is semicommutative. So, the ring $\F_2D_8$ of Example \ref{ex_f2d8} is a minimal abelian reflexive nonsemicommutative ring providing the following theorem.

\begin{theorem}
\label{theo_main2}
A minimal abelian reflexive nonsemicommutative ring has order 256 an example of which is $\F_2D_8$.
\end{theorem}

Next, an example of a minimal abelian reflexive nonsemicommutative ring where the index of nilpotency of the Jacobson radical is 4 is provided. One can see that $\dim_F(J/J^2)=3$, $\dim_F(J^2/J^3)=2$ and $\dim_F(J^3/J^4)=1$. This ring in some sense is just outside the rings described in Propositions \ref{prop_23} and \ref{prop_32}. Furthermore, in Example \ref{ex_f2d8}, the nilpotency of the Jacobson radical is greater than 4 (in $\F_2D_8$, $(1+r),(1+s)\in J(\F_2D_8)$ but $(1+r)^3(1+s)\neq 0$) so, the two rings are non-isomorphic.

\begin{example}
\label{ex2}
Let
\[
R=\RTwo,
\]
$x_0=1$, $x_1=u$, $x_2=v$, $x_3=w$, $x_4=uw$, $x_5=vu$, $x_6=wv$, $x_7=uwv$. Then  $R$ is an $\F_2$ algebra with basis $\{x_0,\dots,x_{7}\}$. Clearly $R$ is local. Since $uv=0$ but $uwv\neq 0$, $R$ is nonsemicommutative. We now show $R$ is reflexive.

Let $a,b\in R$ and assume $aRb=0$. Since $ab=0$, if $a$ or $b$ is a unit, the other is 0 and $bRa=0$. So, assume $a,b\in J(R)$. Then $a=\sum_{i=1}^{7}a_ix_i$ and $b=\sum_{i=1}^{7}b_ix_i$ for some $a_i,b_i\in\F_2$. Since $ab=0$, $aub=0$, $avb=0$ and $awb=0$ it can be shown that $a_1b_2=0$, $a_1b_3=0$, $a_2b_1=0$, $a_2b_3=0$, $a_3b_1=0$, $a_3b_2=0$ and
\[
a_1b_6+a_4b_2+a_2b_4+a_5b_3+a_3b_5+a_6b_1=0.
\]
Let $r\in R$ with $r=\sum_{i=0}^{7}r_ix_i$. Then
\begin{eqnarray*}
bra&=&r_0(b_1a_3uw+b_2a_1vu+b_3a_2wv)+\\
&&r_0(b_1a_6+b_2a_4+b_3a_5+b_4a_2+b_5a_3+b_6a_1)uwv+\\
&&(b_1r_3a_2+b_3r_2a_1+b_2r_1a_3)uwv\\
&=&0.
\end{eqnarray*}
Hence, $R$ is reflexive.
\end{example}

In the previous two examples the residue fields are prime. Next, two examples of minimal abelian reflexive nonsemicommutative rings with non-prime residue fields will be presented.

\begin{example}
\label{ex_1}
Let $\sigma$ be the Frobenius automorphism on $\F_4$ and
\[
R=\ROne.
\]
It will shown that this is a local reflexive nonsemicommutative ring. Let $\alpha\in\F_4\setminus\{0,1\}$. First, $R$ is an 4-dimensional algebra over $\F_4$ with basis $\{1,u,v,u^2\}$ showing $|R| =256$. Furthermore, $J(R)=\id{u,v}$ and $R/J(R)\cong\F_4$, so $R$ is local. Next, since $(u+v)^2=0$ but $(u+v)\alpha(u+v)\neq 0$, $R$ is not semicommutative. Finally, to show $R$ is reflexive, let $a,b\in R$ and assume $aRb=0$. If $a$ or $b$ is a unit, clearly $bRa=0$. Assume $a,b\in J(R)$. Then $a=a_2u+a_3v+a_4u^2$ and $b=b_2u+b_3v+b_4u^2$ for $a_i,b_i\in\F_4$. Now, $aRb=0$, so $0=ab=(a_2b_2+a_3b_3^2)u^2$ and $0=a\alpha b=(a_2\alpha b_2+a_3\alpha^2 b_3^2)u^2$. So, $a_2b_2=a_3b_3^2=0$. Let $r=r_1+r_2u+r_3v+r_4u^2\in R$. Since $J(R)^3=0$, $bra=(b_2r_1a_2+b_3r_1^2a_3^2)u^2=0$. Hence, $R$ is reflexive.
\end{example}

\begin{example}
\label{ex_3}
Let $\alpha\in\F_4\setminus\{0,1\}$. Let $\psi:\F_4[u]\to\F_4[u]$ be the automorphism that maps $u$ to $\alpha u$ and $\alpha$ to $\alpha^2$ and
\[
R=\RThree.
\]
Using analog arguments to those in Example \ref{ex_1}, it can be shown that $R$ is a local reflexive ring which is also a 4-dimensional algebra over $\F_4$, so $|R|=256$. Since $(\alpha^2u+v)^2=0$ but $(\alpha^2u+v)\alpha(\alpha^2u+v)\neq 0$, $R$ is nonsemicommutative.

To see that $R$ is not isomorphic to Example \ref{ex_1}, the orders of the units are explored. Let $r=a+bv\in U(R)$ where $a=a_1+a_2u,b\in{\F_4[u]\over\id{u^2}}$ with $a_1,a_2\in\F_4$. Note $a^2=a_1^2$, $\psi(a^2)=a_1$ and $a_1\neq 0$. Then
\begin{eqnarray*}
r^4=(a+bv)^4&=&a^4+(a^3+a^2\psi(a)+a\psi(a^2)+\psi(a^3))bv\\
&=&a_1^4+(a_1^2a+a_1\psi(a)+a_1a+a_1\psi(a))bv\\
r^6=(a+bv)^6&=&a^6+(a^5+a^4\psi(a)+a^3\psi(a^2)+a^2\psi(a^3)+a\psi(a^4)+\psi(a^5))bv\\
&=&a_1^3+(a_1a+a_1\psi(a)+a_1^3a+a_1^3\psi(a)+a_1^2a+a_1^2\psi(a))bv\\
&=&a_1^3+(a+\psi(a))(a_1+a_1^2+a_1^3)bv.
\end{eqnarray*}
From this it can bee seen that if $a_1=1$ then $r^4=1$ and if $a_1\neq 1$ then $r^6=1$ since $a_1+a_1^2+a_1^3=0$ in this case. So, the order of any unit is at most 6. In Example \ref{ex_1}, the order of $\alpha+u$ is 12. Therefore, $R$ is not isomorphic to the ring in Example \ref{ex_1}.
\end{example}

In \cite{marks_2002}, the ring
\[
{\F_2\id{u,v,w}\over\id{\id{u}^2,\id{v}^2,\id{w}^2,uvw,vwu,wuv}},
\]
was shown to be a 13 dimensional $\F_2$-algebra which is reversible and nonsymmetric. If we take this ring and factor out the ideal generated by $uv$ we obtain the following example which is reflexive and nonsemicommutative.

\begin{example}
Let
\[
R={\F_2\id{u,v,w}\over\id{\id{u}^2,\id{v}^2,\id{w}^2,uv,vwu}},
\]
$x_0=1$, $x_1=u$, $x_2=v$, $x_3=w$, $x_4=uw$, $x_5=vw$, $x_6=vu$, $x_7=wu$, $x_8=wv$, $x_9=uwv$, $x_{10}=wvu$, $x_{11}=vuw$. Then  $R$ is $\F_2$ algebra with basis $\{x_0,\dots,x_{11}\}$. Clearly $R$ is local. Since $uv=0$ but $uwv\neq 0$, $R$ is nonsemicommutative. We now show $R$ is reflexive.

Let $a,b\in R$ and assume $aRb=0$. Since $ab=0$ if $a$ or $b$ is a unit, the other is 0. So assume $a,b\in J(R)$ and $a=\sum_{i=1}^{11}a_ix_i$ and $b=\sum_{i=1}^{11}b_ix_i$ for some $a_i,b_i\in\F_2$. Since $ab=0$,
\[
a_1b_3=0,a_2b_3=0,a_2b_1=0,a_3b_1=0,a_3b_2=0
\]
\[
a_1b_8+a_4b_2=0,a_3b_6+a_8b_1=0,a_2b_4+a_6b_3=0.
\]
and since $awb=0$, $a_1b_2=0$. Let $r\in R$ with $r=\sum_{i=0}^{11}r_ix_i$. Then
\[
b\left(\sum_{i=1}^{11}r_ix_i\right)a=b_1r_3a_2uwv+b_3r_2a_1wvu+b_2r_1a_3vuw=0.
\]
Lastly, we show $br_0a=0$ which is true if $r_0=0$, so assume $r_0=1$. Then
\begin{eqnarray*}
br_0a=ba&=&(b_1a_8+b_4a_2)uwv+(b_3a_6+b_8a_1)wvu+(b_2a_4+b_6a_3)vuw
\end{eqnarray*}
If $a_2=b_1=0$ then $b_1a_8+b_4a_2=0$. If $a_2\neq 0$ then $b_1=b_3=b_4=0$ from the above equations showing $b_1a_8+b_4a_2=0$. If $b_1\neq 0$ then $a_2=a_3=a_8=0$ from the above equations showing $b_1a_8+b_4a_2=0$. A similar analysis shows that $b_3a_6+b_8a_1=0$ and $b_2a_4+b_6a_3=0$, so, $bra=0$. Hence, $R$ is reflexive.
\end{example}

\section{Conclusion}
There are interesting connections between abelian reflexive nonsemicommutative rings studied here and reversible nonsymmetric rings (which are always abelian) studied previously in \cite{marks_2002} and \cite{szabo_2019_4}. This can be seen in the methods used here and those used in \cite{szabo_2019_4} to find the order of a minimal reversible nonsymmetric ring. Alongside the structural details pertaining to finite local rings, another intimate connection that was found here is that the two noncommutative group rings of order 256, namely $\F_2D_8$ and $\F_2Q_8$, are examples of a minimal abelian reflexive nonsemicommutative ring and a minimal reversible nonsymmetric ring, respectively. In \cite{gutan_2004}, reversible group rings were studied. There it was shown that $\F_2Q_8$ is a minimal reversible nonsymmetric group ring. That is of course a consequence of the main result in \cite{szabo_2019_4}. It seems natural then to consider the connections between reflexive group rings and reversible group rings and find out what results in \cite{szabo_2019_4} may pass to abelian reflexive rings.

\bibliographystyle{plain}
\bibliography{../SteveSzaborefs}

\end{document}